\documentclass[11pt]{amsart}
\usepackage{amsfonts,amssymb,amsmath,amsthm}
\usepackage{color}
\usepackage{graphicx}
\usepackage{fullpage}
\usepackage{graphicx}
\usepackage{esint}
\usepackage[usenames,dvipsnames,svgnames,table]{xcolor}
\usepackage[colorlinks=true,
pdfstartview=FitV,linkcolor=ForestGreen,citecolor=ForestGreen,
urlcolor=blue]{hyperref}
\usepackage{enumitem}
\usepackage{stix}
\usepackage{bbm}
\newcommand{\norm}[1]{\left\| #1 \right\|}

\numberwithin{equation}{section}

\newtheorem{theorem}{Theorem}
\newtheorem{proposition}[theorem]{Proposition}
\newtheorem{lemma}[theorem]{Lemma}
\newtheorem{corollary}[theorem]{Corollary}

\theoremstyle{definition}
\newtheorem{definition}[theorem]{Definition}

\theoremstyle{remark}
\newtheorem{remark}[theorem]{Remark}

\newcommand\R{{\mathbb R}}

\newcommand{\Cdiv}{C_{\mathrm{div}}}

\newcommand{\Cgeh}{C_{\mathrm{G}}}

\newcommand{\cA}{\mathcal A}
\newcommand{\cB}{\mathcal B}

\newcommand{\cG}{\mathcal G}

\newcommand{\cQ}{\mathcal Q}

\newcommand{\cS}{\mathcal S}

\newcommand{\sg}{\mathscr{g}}
\newcommand{\sh}{\mathscr{h}}

\def\fg{\mathfrak{g}}
\def\fh{\mathfrak{h}}
\def\bg{\mathbf{g}}
\def\bh{\mathbf{h}}

\def\eps{{\varepsilon}}

\newcommand\dz{\, \mathrm{d} z}

\newcommand{\dd}{{\, \mathrm d}}

\newcommand{\fa}{\forall \,}

\let\oldmarginpar\marginpar
\renewcommand\marginpar[1]{\-\oldmarginpar[\raggedleft\footnotesize #1]%
{\raggedright\footnotesize #1}}

\title[Gehring' Lemma for kinetic Fokker-Planck
equations]{Gehring's  Lemma for kinetic Fokker-Planck equations}

\author{Jessica Guerand}

\address[Jessica Guerand]{IMAG, Bâtiment 9 sur le Campus du
  Triolet, 499-554 Rue du Truel, 34090 Montpellier, France}

\email{jessica.guerand@umontpellier.fr}

\author{Cyril Imbert}

\address[Cyril Imbert]{Centre National de la Recherche
  Scientifique \& \'Ecole normale sup\'erieure - Université PSL,
  D\'epartement de Math\'ematiques et Applications, UMR 8553, 45
  rue d’Ulm, 75005 Paris, France}

\email{Cyril.Imbert@ens.psl.eu}

\author{Cl\'ement Mouhot}

\address[Cl\'ement Mouhot]{University of Cambridge, Department of
  Pure Mathematics and Mathematical Statistics, Wilberforce Road,
  Cambridge CB3 0WA, United Kingdom}

\email{c.mouhot@dpmms.cam.ac.uk}

\date{\today}

\begin{document}

\begin{abstract}
  In this article, we establish a ``Gehring lemma'' for a real
  function satisfying a reverse H\"older inequality on all
  ``kinetic cylinders'' contained in a large one: it asserts that
  the integrability degree of the function improves under such an
  assumption. The kinetic cylinders are derived from the
  non-commutative group of invariances of the Kolmogorov
  equation. Our contributions here are (1) the extension of
  Gehring's Lemma to this kinetic (hypoelliptic) scaling used to
  generate the cylinders, (2) the localisation of the lemma in
  this hypoelliptic context (using ideas from the elliptic
  theory), (3) the streamlining of a short and quantitative
  proof. We then use this lemma to establish that the velocity
  gradient of weak solutions to linear kinetic equations of
  Fokker-Planck type with rough coefficients have Lebesgue
  integrability strictly greater than two, while the natural
  energy estimate merely ensures that it is square
  integrable. Our argument here is new but relies on
  Poincaré-type inequalities established in previous works.
\end{abstract}

\maketitle
 
\tableofcontents

\section{Introduction}

In this note, we are interested in establishing a version of
Gehring's Lemma in the context of kinetic equations of
Fokker-Planck type. Such equations write
\begin{equation}
  \label{e:main}
  (\partial_t  + v \cdot \nabla_x) f = \nabla_v \cdot ( \cA \nabla_v
  f) + \cB \cdot \nabla_v f + \cS,  \quad (t,x,v) \in Q
\end{equation}
for $0 < \lambda < \Lambda$, some function $\cA$ valued in the
set of real symmetric matrices whose eigenvalues lie in a given
interval $[\lambda, \Lambda]$, some real-valued vector field
$\cB$ bounded in $Q$ (\textit{i.e.}
$\|\cB\|_{L^\infty(Q)} \le \Lambda$) and some real-valued source
term $\cS \in L^2 (Q)$. The unknown function $f$ is also
real-valued and the equation is posed in some cylinder
$Q \subset \R \times \R^d \times \R^d$.

A ``Gehring Lemma'' roughly says that a function $f \in L^q(B)$
for some ball $B$ enjoys in fact better Lebesgue integrability as
soon as the averaged $L^q$-norm of $f$ in any ball $B' \subset B$
is controlled by the averaged $L^1$-norm of $f$ in $B'$, with
uniform constant. It was first proved by Gehring~\cite{MR0402038}
motivated by open problems related to quasiconformal mappings,
and then adapted and used by various authors to study gain
integrability on the gradient of solutions of elliptic equations,
see for instance the following early contributions:
\cite{MR0417568,GiaquintaModica}.
In order to establish a similar result for solutions of parabolic
equations \cite{GiaquintaStruwe,zbMATH01285988,MR1754355}, it is necessary to
extend Gehring's Lemma by replacing Euclidian balls with
parabolic cylinders: $Q_r (t,v) = (t-r^2,t] \times B_r (v)$ which
do respect the invariances of such parabolic equations: the
translations in $(t,v)$ and the parabolic scalings
$(t,v) \mapsto (r^2t,r v)$ for $r>0$.

The class of kinetic Fokker-Planck equations~\eqref{e:main} is
invariant by the slightly more complicated scalings
$r \mapsto (r^2 t, r^3 x, r v)$ for $r>0$. It is preserved by
translations in $(t,v)$ but not in $x$, instead it enjoys
Galilean invariance: given
$z_0 = (t_0,x_0,v_0) \in \R\times\R^d\times \R^d$ and
$z = (t,x,v)$, we define
\begin{equation*}
  z_0 \circ z = (t_0+t,x_0+x + t v_0, v_0 + v).
\end{equation*}
If $f$ solves~\eqref{e:main}, then $f (z_0 \circ z)$ solves
an equation of the same form (in particular, the eigenvalues of
the diffusion matrices still lie in $[\lambda,\Lambda]$).
Because of these two invariances, it is natural to define the
so-called \emph{kinetic cylinders} by $Q_r (z_0) = z_0 \circ
Q_r$ with \(Q_r = (-r^2,0] \times B_{r^3} \times B_r .\) More
explicitly,
\begin{equation*}
Q_r(z_0)=\left\{z=(t,x,v): -r^2<t-t_0\leq 0,
  |x-x_{0}-(t-t_0)v_0|<r^3, |v-v_{0}|<r \right\}.
\end{equation*}
Our first main result is a Gehring Lemma extended to this
setting. For a kinetic cylinder $Q$ and a function
$F \in L^1 (Q)$, we let $\fint_Q F$ denote $|Q|^{-1} \int_Q F$
where $|Q|$ stands for the volume of $Q$.
\begin{theorem}[Localised Gehring's Lemma with the kinetic  scalings] \label{thm:gehring} Let $\gamma>1$. Consider
  non-negative functions $g \in L^q(Q_{\gamma})$ and  $h \in L^\sigma(Q_{\gamma})$  for some $\sigma>q >1$.
  Assume that there is  $b >1$ and $\theta \in (0,1)$ so that for all $R>0$ and   $z_0\in Q_{\gamma}$ such that  $Q_{\gamma R}(z_0) \subset Q_{\gamma}$
  \begin{equation}
    \label{eq:reverse-holder}
    \fint_{Q_{R}(z_0)} g^q  \le b \left\{ \left( \fint_{Q_{\gamma
            R}(z_0)} g  \right)^q+\fint_{Q_{\gamma  R}(z_0)}
      h^q  \right\}+ \theta\fint_{Q_{\gamma  R}(z_0)} g^q .
  \end{equation}
  If $\theta<\theta_0:= [2 (75)^{4d+2} 4^q]^{-1}$  and $q < p^*:=\min\left(\sigma,\frac{bq-\theta_0}{b-\theta_0}\right)$ then $g \in L^p(Q)$ for
  $p \in [q,p^*)$, and this norm is controlled by
  \begin{equation}
    \label{eq:improved-int}
  \left( \fint_{Q_{1}} g^p  \right)^{1/p} \le \Cgeh \left(
    \left( \fint_{ Q_{\gamma}} g^q  \right)^{1/q}+   \left(
      \fint_{ Q_{\gamma}} h^p  \right)^{1/p} \right)      
\end{equation}

where $\Cgeh$ depends only on $b,\theta,p,q,\gamma, d$. 
\end{theorem}
Our second main result is an application of the previous lemma to
kinetic equations of Fokker-Planck type, and proves the gain of
integrability on the velocity gradient of their solutions.
\begin{theorem}[Gain of integrability of the velocity gradient]
  \label{t:gain velocity gradient}
  Let $f$ be a weak solution to~\eqref{e:main} in $Q$ with
  $\cS \in L^{2+\eps_0} (Q)$ for some $\eps_0>0$. Let $\gamma>1$
  and $R>0$. There exist two constants
  $\varepsilon \in (0,\eps_0)$ and $C>0$, both depending on
  $d, \lambda, \Lambda, \gamma, R,\eps_0$, such that, for all
  $z\in Q$ such that $Q_{\gamma R}(z)\subset Q$, we have,
  \begin{equation*}
    \norm{\nabla_v f}_{L^{2+\varepsilon}(Q_{R}(z))}\le C
    \left(\norm{\nabla_v f}_{L^{2}(Q_{\gamma R}(z))} +
      \norm{\cS}_{L^{2+\eps}(Q_{\gamma R} (z))} \right).
  \end{equation*}
\end{theorem}
\begin{remark}
  Weak solutions of \eqref{e:main} are defined in
  Section~\ref{s:prelim} (see Definition~\ref{defi:weak-sol}).
\end{remark}

\subsection{Comments on the main results and their proofs}

Theorem~\ref{thm:gehring} asserts that an $L^q$ function in a
kinetic cylinder $Q_\gamma$ with $\gamma>1$ is in fact in $L^p$
with $p>q$ in the smaller kinetic cylinder $Q_1$ as soon as a
family of reverse H\"older inequalities hold true at every
scales, where the scales used stem from the invariances of the
Kolomogorov equation. It is an \emph{interior} regularity
estimate since, on the one hand, reverse H\"older inequalities
relate Lebesgue norms in interior cylinders $Q_R(z)$ with norms
in larger cylinders $Q_{\gamma R} (z)$, and on the other hand the
improvement only holds in $Q_1 \subset Q_\gamma$.

The global strategy of the original proof in~\cite{MR0402038} is
to \emph{transfer} reverse H\"older inequalities from cylinders,
whose geometry is independent of the function $g$, onto the
superlevel sets of $g$. This reduces then the proof to the case
of dimension $d=1$, and one concludes by explicit
calculations. However, this strategy works in the whole space,
and the first step of the proof is to reduce to the case where
the function $g$ is defined on the whole space
$\R \times \R^d \times \R^d$. We do so thanks to a
\emph{localization procedure}, following an idea appearing
in~\cite{Iwaniec98} in the case of Gehring's Lemma for elliptic
equations, i.e. for the standard Euclidean scalings. 

In order to transfer the reverse H\"older inequalities from the
function to its superlevel sets, we use a standard covering
argument. This requires an extension of the classical Vitali
lemma and a differentiation theorem à la Lebesgue for kinetic
cylinders. Both results are already proved in~\cite[Lemma~10.5 \&
Theorem~10.3]{ImbertSilvestre}: we recall their statements in
Section~\ref{s:prelim} and refer to~\cite{ImbertSilvestre} for
their proofs.  \medskip

The Gehring lemma is often used to obtain a gain of integrability
on the gradients of solutions to elliptic or parabolic equations
(see below). It naturally starts with an estimate of Cacciopoli
type, \textit{i.e.} the basic local energy estimate for such
equations. It relates the gradient of the solution with its $L^2$
norm, say. In order to derive a local reverse H\"older inequality
for the gradient, it is thus necessary to estimate from above the
$L^2$ norm of the solution with the $L^q$ norm of the gradient
for some $q \in [1,2)$. In order to do so, one typically uses a
gain of integrability on the solutions (not their gradient). In
the simplest case of elliptic equations, the gain of integration
on solutions follows from the energy estimate and the Sobolev
embedding, and the control of the $L^q$ norm of
the solutions by that of their gradient follows from the
Poincaré-Wirtinger inequality. 

Several ingredients are needed in order to extend this strategy
of proof from the elliptic setting to kinetic equations of
Fokker-Planck type. Cacciopoli estimate and the gain of
integrability were already proved in~\cite{MR2068847}. They play
a key role in the proof of the Harnack inequality
in~\cite{gimv}. The other ingredient we need for applying
Gehring's Lemma is an inequality of Poincaré type in $L^q$ with
$q \in (1,2)$. The difficulty lies in the fact that only the
gradient in the velocity variable should appear in the right hand
side of the desired inequality. Several Poincaré inequalities
were proved in the context of kinetic equations. Such an estimate
is key in the derivation of the local H\"older estimate in
\cite{MR2530175}. A more accurate Poincaré inequality was derived
in \cite{am19}. We largely follow \cite{am19} in order to derive
the Poincaré inequality we need. The reader is also referred to
\cite{zbMATH07559605,zbMATH07750909} for other Poincaré
inequalities in the context of kinetic Fokker-Planck equations,
and~\cite{dietert2023quantitative} for their extensions in the
case of boundary conditions.

We note that the gain of integrability of the velocity gradient
of solutions to kinetic Fokker-Planck equations was announced in
\cite{gimv}. However it was based on a Gehring Lemma stated with
other type of cylinders and, as a consequence, the proof for the
estimate on the velocity gradient was incomplete.

\subsection{Short review of literature}

Gehring first discovered in~\cite{MR0402038} the fact that the
integrability of a function improves when it satisfies a family
of reverse Hölder inequalities. It allowed him to unlock the
study of regularity of so-called quasiconformal mappings for
which reverse H\"older inequalities were known to hold true. This
finding turned out to be a powerful tool in the study of
regularity of both (quasi)-minimizers of functionals and
solutions to partial differential equations. Many authors
established variants of the original Gehring lemma in different
contexts. In~\cite{GiaquintaModica}, a local version of the
Gehring lemma was
obtained. In~\cite{modica1985quasiminimi,zbMATH00691575}, the
authors added to the Lebesgue measure a weight that satisfies a
doubling condition. In~\cite{Iwaniec98}, the Lebesgue spaces are
replaced by Orlicz spaces (see also~\cite{MR2924890,MR3879985}).

The other direction in which Gehring's Lemma can be extended,
which is the focus of this paper, is when the reverse Hölder
inequalities are assumed in cylinders that are not necessarily
Euclidian balls. Motivated by the study of hypoelliptic equations
involving square Hörmander operators of type A, i.e. of the ``sum
of squares'' form $\sum_i X_iX_i^*$, a lemma of Gehring type is
proved in this direction in \cite{MR1288498} proved a Gehring
lemma when the reversed Hölder inequalities are assumed in
``balls'' defined by the pseudo-distance created by the vector
fields $X_i$. We note that, in fact, instead of considering a
family of reverse H\"older inequalities, the sufficient condition
for getting an improvement of integrability can be expressed in
terms of two maximal functions associated with those balls. Even
more general operators can be handled by considering reverse
H\"older inequalities on so-called metric balls \cite{MR1398170}
(see also \cite{MR2173373}). However this class of operators does
not include~\eqref{e:main}, which is in fact the prototypical
case of Hörmander operators of type B, due to the presence of the
first order transport operator. Such kinetic equation is moreover
similar to the Kolmogorov equation~\cite{MR1503147} that was the
initial motivation of Hörmander's seminal paper~\cite{MR222474}.

In \cite{zbMATH01285988,MR1754355}, the authors are interested in
the regularity of solutions to parabolic quasilinear systems of
equations. They show that it is possible to derive reverse
H\"older inequalities for the gradients of such solutions if
Euclidian balls are replaced by parabolic cylinders. They show
that this family of inequalities yields again an improvement of
the integrability of these gradient functions.

More recently, non-local versions of the Gehring lemma were
derived \cite{MR3283394,MR3336922}. Local reverse H\"older
inequalities are replaced by estimates of Cacciopoli type for
linear integro-differential equations with kernels comparable
with the one of the fractional Laplacian.  A short proof using
complex interpolation theory were recently given in
\cite{MR4167265}. We note that \emph{discrete} versions of the
Gehring lemma have also been devised, see for
instance~\cite{MR4033780} and the references therein.

We conclude this short review of literature by mentioning the
existence of the nice survey article~\cite{Iwaniec98} about the
Gehring lemma. As mentioned above, we borrowed from this work the
idea of using a cut-off function in the localization process in
the proof of Theorem~\ref{thm:gehring} (see the function $\zeta$
from Subsection~\ref{sub:localization}).

\subsection{Organization and notation}

We gather in Section~\ref{s:prelim} known results that are needed
in the proofs of the main results. Section~\ref{s:reversed} is
devoted to the proof of Gehring's Lemma
(Theorem~\ref{thm:gehring}) while Section \ref{sec:Gain gradient}
contains the proof of the gain of integrability for the velocity
gradient (Theorem~\ref{t:gain velocity gradient}). The variable
$z$ denotes $(t,x,v)$. The open Euclidian ball centered at the
origin of radius $r>0$ is denoted by $B_r$. If the center is
$v_0$, then it is denoted by $B_r (v_0)$. Cylinders centered at
the origin are denoted by
$Q_r:=(-r^2,0] \times B_{r^3} \times B_r$, and $Q_r (z_0)$
denotes $z_0 \circ Q_r$.

\section{Preliminaries}
\label{s:prelim}

\subsection{The Vitali lemma}
We first state and prove a Vitali lemma for kinetic cylinders similar to the one in \cite[Lemma~10.5]{ImbertSilvestre} but whose definition is slightly different. To state this
covering result, we need to introduce those kinetic cylinders that are
scaled by a factor $5$ and shifted in time (note that this
definition is \emph{not} simply the hypoelliptic scaling by a
factor $5$) 
\begin{equation}
  \label{def:stacked cylinder}
  5Q_R (z_0) =   Q_{5R} \left(z_0 \circ \left(\tau_{t_0,R},0,0 \right)
  \right) \quad \mbox{ with }  \tau_{t_0,R}:=\min(-t_0,12 R^2).
\end{equation}
More explicitely, we have
\begin{equation*}
  5Q_{R}(z_0)=\bigg\{(t,x,v): -25R^2+\tau_{t_0,R} <t-t_0\leq \tau_{t_0,R},  |x-x_{0}-(t-t_0)v_0|<(5R)^3, |v-v_{0}|<5R \bigg\}.
\end{equation*}
The shift in time ensures that $z_0$ lies in the interior of
$5Q_R(z_0)$ unless $t_0=0$. It is necessary to consider such a shift in order to
get a covering à la Vitali.
\begin{lemma}[Vitali lemma]
  \label{lem:Vitali}
  Let $(Q_j)_{j\in J}$ be a collection of kinetic cylinders  in $(-\infty,0] \times \mathbb{R}^d\times \mathbb{R}^d$ with
  bounded radii. Then there exists a disjoint countable
  subcollection $(Q_i)_{i \in I}$ where $I\subset J$ such that
  $$\bigcup_{j\in J} Q_{j}  \subset \bigcup_{i\in I} 5Q_{i},$$
  where $5Q_{i}$ is defined in~\eqref{def:stacked cylinder}.
\end{lemma}
\begin{proof}
Vitali lemma applies thanks to the following property for $z_1=(t_1,x_1,v_1)$ and $z_2=(t_2,x_2,v_2)$ with $t_1, t_2 \leq 0$ and $r_1, r_2>0$
\begin{equation*}
Q_{r_1}(z_1)\cap Q_{r_2}(z_2) \neq \emptyset \quad \mbox{ and } r_1\leq 2r_2 \quad \Rightarrow Q_{r_1}(z_1)\subset Q_{5r_2}(z_2\circ (\tau_{t_2,r_2},0,0)).
\end{equation*}
Take $z_0=(t_0,x_0,v_0)$ in the intersection and $z=(t,x,v)\in Q_{r_1}(z_1)$. We justify  that $z\in Q_{5r_2}(z_2\circ (\tau_{t_2,r_2},0,0))$ thanks to the following inequalities. Inequality $|v-v_2|\leq 5r_2$ comes naturally. 
Then since $t\leq 0$, $t-t_2\leq -t_2$ and $t-t_2= (t-t_1)+(t_1-t_0)+(t_0-t_2)\leq r_1^2\leq 4r_2^2$ so that $t-t_2\leq \tau_{t_2,r_2}$. Moreover, $t-t_2\geq -r_1^2-r_2^2\geq -5r_2^2\geq -25r_2^2+\tau_{t_2,r_2}$. Finally
\begin{align*}
|x-x_2-(t-t_2)v_2|\leq &|x-x_1-(t-t_1)v_1|+|x_1-x_0+(t_0-t_1)v_1|+|x_0-x_2-(t_0-t_2)v_2|\\
&+(|t-t_1|+|t_1-t_0|)(|v_1-v_0|+|v_0-v_2|)\leq  41 r_2^3 \leq  (5r_2)^3,
\end{align*}
which concludes the proof.
\end{proof}


\subsection{Lebesgue's differentiation}

Then we recall Lebesgue's differentiation theorem in the context of kinetic cylinders, see \cite[Theorem 10.3]{ImbertSilvestre}.
\begin{lemma}[Lebesgue differentiation]
  \label{lem: lebesgue diff}
  Let $f\in L^{1}(\Omega)$ where $\Omega$ is an open set of
  $\R^{2d+1}$. Then for almost every $z=(t,x,v)\in \Omega$,
  \begin{equation*}
    \lim\limits_{r\rightarrow 0^{+}} \fint_{ Q_{r}(z_r)}
    |f(z')-f(z)| \mathrm{d}t'\mathrm{d}x'\mathrm{d}v'=0
  \end{equation*}
  where the sequences $z_r \in \Omega$ and $r>0$ are so that
  $z \in Q_r (z_r) \subset \Omega$ for all $r>0$.
\end{lemma}
\begin{remark}
  In \cite[Theorem 10.3]{ImbertSilvestre}, the result is stated
  for $z_r = z$. The proof follows the classical argument: with a
  Vitali covering lemma at hand, prove first a weak $L^1$
  inequality for a maximal function and then argue by
  approximation with continuous functions.  This proof does not
  use the fact that $z_r =z$.
\end{remark}

\subsection{Weak solutions}

There are several notions of weak solutions for kinetic
Fokker-Planck equations~\eqref{e:main}.  The weakest notion
appears in \cite{auscher:hal-04519638}, it superseeds for
instance the notion introduced in
\cite{gimv,zbMATH07559605,zbMATH07750909}. It is stated for
$(x,v) \in \R^d \times \R^d$ but it can be localized by
multiplying by a cut-off function. We also give the definition in
the inhomogeneous setting ($L^2_{t,x}H^1_v$ instead of
$L^2_{t,x} \dot{H}^1_v$) in order to remove the restriction
$1 < d/2$ -- see \cite[\S~7.2]{auscher:hal-04519638}.
\begin{definition}[Weak solutions --
  \cite{auscher:hal-04519638}]
  \label{defi:weak-sol}
  Let $\mathcal{U} = (a,b) \times \Omega_x \times \Omega_v$ with
  $- \infty < a < b \le +\infty$ and $\Omega_x, \Omega_v$ two
  open sets of $\R^d$.  A function $f \colon \mathcal{U} \to \R$
  is a \emph{weak solution} of \eqref{e:main} if
  \( f \in L^2 ((a,b) \times \Omega_x, H^1 (\Omega_v)) \) and if
  $f$ satisfies \eqref{e:main} in the sense of distributions in
  $\mathcal{U}$. 
\end{definition}

\subsection{Energy estimate}

We next recall the natural energy estimate associated with Eq.~\eqref{e:main}. The reader is referred to \cite{MR2068847,gimv} or \cite[Proposition 9]{zbMATH07559605} for a proof.
\begin{proposition}[Energy estimate]\label{prop:EE}
Let $f$ be a weak sub-solution of \eqref{e:main} in
$Q$. Then for all cylinders $Q_r(z_0)$ and $Q_R(z_0)$ with
$0<r<R$ such that $Q_{R}(z_{0})\subset Q$,
\begin{equation*}
 \int_{Q_r(z_0)} |\nabla_v f |^2
  \le C \left(\overline{C} \int_{Q_R(z_0)} f^2 + \int_{Q_R(z_0)} S^2  \right)
\end{equation*}
where $Q_r^t (z_0) = \{(x,v)\in\R^2: (t,x,v)\in Q_r(z_0) \} $ and
$C$ only depends on the dimension $d$, $\lambda$ and $\Lambda$, and 
$\overline{C}$ only depends on $R$ and $r$.
\end{proposition}

\subsection{Gain of integrability of the solution}

In this subsection, we recall a result of gain of integrability
of the solution \cite[Proposition 11]{zbMATH07559605} with
$p\in [2,2+1/d)$ which is a tool for the proof of Theorem
\ref{t:gain velocity gradient}. 
\begin{proposition}[Gain of integrability of the solution]
  \label{prop:Gain sol}
  Let $f$ be a weak solution of \eqref{e:main} in $Q$.  Then for
  all cylinders $Q_r(z_0)$ and $Q_R(z_0)$ with $0<r<R$,
  $z_{0}\in Q$ such that $Q_{R}(z_{0})\subset Q$, $f$ satisfies
  \begin{equation*}
    \norm{f}_{L^{p}(Q_r(z_0))}^2\leq
    C\left(\overline{C}^2\norm{f}_{L^{2}(Q_R(z_0))}^2 +
      \overline{C}\|S\|_{L^2(Q_R(z_0))}^2 \right),
  \end{equation*}
  where $p=2+\frac{1}{d}>2$, $C$ only depends on the $d$,
  $p$, $\lambda$, $\Lambda$ and $\overline{C}$ only depends on $R$,$r$ and $v_0$.
\end{proposition}

\section{Proof of Gehring's Lemma}
\label{s:reversed}

\subsection{Localization} 
\label{sub:localization}

In order to deal with the boundary of $Q_\gamma$, we use a
localisation function in the spirit of Iwaniec's work
\cite{Iwaniec98} for the elliptic case.  We define for
$z=(t,x,v) \in \R^{1+2d}$,
\begin{equation*}
  \zeta(t,x,v)= \frac{1}{2}\min\left(\frac{(\gamma-|v|)_+}5,
    \left(\frac{(\gamma^2+t)_+}{{ 13}}\right)^{1/2},
    \left(\frac{(\gamma^3-|x|)_+}{{  25} \gamma}\right)^{1/2}
  \right).
\end{equation*}
Let us first give a useful estimate on $\zeta$ in $ Q_{r}(z_0)$.
\begin{lemma}[Estimate on $\zeta$]
  \label{l:zeta_estimate}
 If $z\in Q_{r}(z_0)\subset Q_{\gamma}$ for $r>0$ and $z_0\in Q_{\gamma}$ then $|\zeta(z)-\zeta(z_0)|\leq   \frac{r}{2}$.
\end{lemma}
\begin{proof}
 The idea guiding the calculation is that $\zeta$ increases with the
  distance to the parabolic boundary $\partial_p Q_\gamma$ of $Q_\gamma$. Hence,
when $z_0$ is not too far from $\partial_p Q_\gamma$ and since $z$ is close to $z_0$, so is
  $z$.
Here is the detailed calculation for $z =(t,x,v)$:
 \begin{align*}
    & \frac{(\gamma -|v|)_+}{5} \le \frac{(\gamma -|v_0|)_+}{5} + \frac{r}5 \\
    & \left(\frac{(\gamma^2+t)_+}{{ 13}}\right)^{\frac12}
     \le \left(\frac{(\gamma^2+t_0)_++r^2}{{ 13}}\right)^{\frac12} \le
      \left(\frac{(\gamma^2+t_0)_+}{{ 13}}\right)^{\frac12}+\frac{r}{ 
      \sqrt{13}}
  \end{align*}
   and finally 
  \begin{align*}
    \left(\frac{(\gamma^3 -|x|)_+}{{  25} \gamma}\right)^{\frac12}
    & \le \left(\frac{(\gamma^3 -|x_0|)_+}{{  25} \gamma} + \frac{r^3 +
      |v_0|r^2}{{  25} \gamma} \right)^{\frac12} \\
    & \le \left( \frac{(\gamma^3 -|x_0|)_+}{{  25}\gamma} +
      \frac{{ 2} r^2}{{  25}}\right)^{\frac12} \qquad \text{(we used that
      $|v_0|\le \gamma$ and $r \le \gamma$)} \\
    & \le \left( \frac{(\gamma^3 -|x_0|)_+}{{  25}
      \gamma}\right)^{\frac12} + {  \frac{\sqrt{2}}{5}}r,
  \end{align*}
  which gives,
  \begin{align*}
  \zeta(z) \leq \zeta(z_0) +\frac{r}{  5 \sqrt 2 }.
  \end{align*}
  A similar computation gives 
   \(
  \zeta(z) \geq \zeta(z_0) -\frac{r}{  5 \sqrt 2}.
  \)
\end{proof}
\begin{lemma}[Shifted and scaled cylinders]
  \label{l:shifted}
  If $z_0=(t_0,x_0,v_0)\in Q_{\gamma}$  and $r \leq 2\zeta(z_0)$, then $5 Q_{r}(z_0)\subset Q_{\gamma}$.
\end{lemma}
\begin{proof}
  Let $r \le \zeta_0 :=2\zeta(z_0)$. The
  condition $z \in 5 Q_r (z_0)$ is equivalent to,
  \begin{equation*}
    \begin{cases}
      |v-v_0| < 5r, \\
      -25 r^2+\tau_{t_0,r} < t-t_0 \le \tau_{t_0,r}, \\
      |x-x_0 - (t-t_0) v_0 | < (5r)^3.  
    \end{cases}
  \end{equation*}
  Then we have first $|v| < |v_0|+ 5r$ and the condition
  $r \le \zeta_0 \le (\gamma - |v_0|)_+/5$ implies
  $|v| < \gamma$. Second we have $t \le t_0 + \min(-t_0,12r^2)$ which implies
  $t \le 0$.

  Third we have $t > t_0 - 25 r^2+\min(-t_0,12r^2)$ and the condition
  $r^2 \le \zeta_0 ^2 \le \min((\gamma^2 + t_0)/13,\gamma^2/25)$ implies
  $t > -\gamma^2$. Indeed, if $-t_0 \le 12 r^2$, then $t > -25 r^2 \ge - \gamma^2$.
  And if $-t_0 \ge 12 r^2$, then $t > t_0 - 13 r^2 \ge - \gamma^2$.
  Remark that this implies that $|t-t_0 | < 25 r^2$.

  Fourth and last we have
  $|x| < |x_0| + |t-t_0| |v_0| \le |x_0| +   25 r^2 \gamma$ and the
  condition
  $r^2 \le \zeta_0 ^2 \le (\gamma^3-|x_0|)_+/(  25 \gamma)$ implies
  $|x| < \gamma^3$.
  Hence the conclusion $5 Q_{r}(z_0)\subset Q_{\gamma}$ holds.
\end{proof}

With such a localization function and such cylinders at hand, we
scale the functions $g$ and $h$.

\begin{lemma}[Localization]
  \label{l:localization}
  Given $\gamma>1$ and non-negative functions
  $g \in L^q(Q_{\gamma})$ and $h \in
  L^\sigma(Q_{\gamma})$ for some $\sigma>q >1$ such that there is $b
  >1$ and $\theta \in (0,1)$ so that for all $R>0$ and
  $Q_{\gamma R}(z_0) \subset
  Q_{\gamma}$ the inequality~\eqref{eq:reverse-holder} holds. Let
  $\bg$ and $\bh$ be defined on $\R \times \R^d \times
  \R^d$ by the following formulas,
  \begin{equation*}
    \bg (z):=\frac{| Q_{\zeta
        (z)}|^{\frac1q}}{C_0}\frac{g(z)}{\norm{g}_{L^{q}(
        Q_{\gamma})}}, \qquad \bh (z) := \frac{| Q_{\zeta
        (z)}|^{\frac1q}}{C_0}\frac{h(z)}{\norm{g}_{L^{q}(
        Q_{\gamma})}}
  \end{equation*}
  with $C_0 = 2^{\frac{4d+2}q}$.
  \begin{itemize}
  \item[(i)] The function $\bg$ satisfies, for all   $z_0\in Q_{\gamma}$ and $r> (2/3) \zeta (z_0)$ such that
    $Q_r (z_0) \subset Q_\gamma$,
    \begin{equation}
      \label{cond initial normalisation}
      \fint_{\, Q_r(z_0)} \bg^q \leq 1. 
    \end{equation}
  \item[(ii)] The functions $\bg$ and $\bh$ satisfy, for all
    $z_0=(t_0,x_0,v_0)\in Q_{\gamma}$ and $r\in (0,\zeta (z_0))$,
    \begin{align}
      \label{eq:holdinv}
      \fint_{Q_{r}(z_0)} \bg^q  \le \bar b \left\{ \left(
      \fint_{Q_{\gamma r}(z_0)} \bg  \right)^q
      +\fint_{Q_{\gamma r}(z_0)} \bh^q  \right\}+ \bar
      \theta\fint_{Q_{\gamma r}(z_0)} \bg^q 
    \end{align}
    with $\bar b = C b$ and $\bar \theta = C \theta$ with $C=3^{4d+2}$.
\end{itemize}
\end{lemma}
\begin{proof}
  To prove (i), we estimate the mean of $\bg^q$ on $Q_r (z_0)$
  with $r>(2/3) \zeta (z_0)$,
  \begin{align*}
    \fint_{Q_r(z_0)} \bg^q
    & = \frac1{C_0^q} \frac1{\|g\|^q_{L^q(Q_\gamma)}}\int_{Q_r
      (z_0)} \frac{|Q_{\zeta (z)}|}{|Q_r (z_0)|}  g(z)^q \dz \\
    & =  \frac1{\|g\|^q_{L^q(Q_\gamma)}}\int_{Q_r (z_0)}
      \frac1{C_0^q} \left(\frac{\zeta (z)}{r}\right)^{4d+2}
      g(z)^q \dz.
  \end{align*}
  in order to obtain~\eqref{cond initial normalisation}, it is enough to
  check that we can choose $C_0>0$ such that
  \begin{equation*}
    \forall z \in Q_r (z_0), \quad \frac1{C_0^q}\left(\frac{\zeta
        (z)}{r}\right)^{4d+2} \le 1.
  \end{equation*}
  If we pick $C_0 = 2^{\frac{4d+2}q}$, the previous inequality holds true as soon as
  \begin{equation}
    \label{e:zeta-r}
    \forall z \in Q_r (z_0), \quad \zeta (z)  \le 2 r. 
  \end{equation}  
  It is a consequence of the fact that $r > (2/3) \zeta (z_0)$, $z \in Q_r (z_0)$ and Lemma \ref{l:zeta_estimate} since  it implies that
    for all $z \in Q_r (z_0) \subset Q_\gamma$, we have $\zeta (z)    \le \zeta(z_0) + r/2$.

  \bigskip
  We now check \eqref{eq:holdinv}.  
 We first remark thanks to Lemma \ref{l:zeta_estimate} that for $z\in Q_r(z_0)$ when $r\leq \zeta (z_0)$, we get $\frac{\zeta(z_0)}{2}\leq\zeta(z) \leq \frac{3\zeta(z_0)}{2}$ . 
   This implies
  \begin{equation}
    \label{e:0}
    C_0^q \|g\|_{L^q (Q_\gamma)}^q \fint_{Q_r (z_0)} \bg^q \le
    \left(\frac{3}{2}\right)^{4d+2}|Q_{\zeta(z_0)}| \fint_{Q_r (z_0)} g^q,
  \end{equation}
  \begin{align}
     \noindent
    \fint_{Q_{\gamma r}(z_0)} g^q 
    \label{e:1}
    & \le   2^{4d+2} 
      \|g\|_{L^q(Q_\gamma)}^q \frac{C_0^q}{|Q_{\zeta(z_0)}|} \fint_{Q_{\gamma r}(z_0)}
      \bg^q.   
  \end{align}
    The same computation yields,
  \begin{equation}
    \label{e:2}
    \fint_{Q_{\gamma r}(z_0)} h^q  \le 2^{4d+2}
    \|g\|_{L^q(Q_\gamma)}^q  \frac{C_0^q}{|Q_{\zeta(z_0)}|} \fint_{Q_{\gamma r}(z_0)} \bh^q. 
  \end{equation}
  We next estimate the mean of $g$,
  \begin{equation}
    \label{e:3}
    \left(\fint_{Q_{\gamma r}(z_0)} g \right)^q \le
    2^{4d+2}
    \|g\|_{L^q(Q_\gamma)}^q \frac{C_0^q}{|Q_{\zeta(z_0)}|}\left(\fint_{Q_{\gamma r}
        (z_0)} \bg \right)^q.
  \end{equation}
  Combining \eqref{eq:reverse-holder} with \eqref{e:0},
  \eqref{e:1}, \eqref{e:2} and \eqref{e:3} yields \eqref{eq:holdinv} with $\bar b = C b$ and
  $\bar \theta = C \theta$ and $C=3^{4d+2}$.
\end{proof}

\subsection{Covering of the superlevel sets of $\bg$}

\begin{lemma}[Covering of superlevel sets]
  \label{l:covering}
  Given $\gamma>1$ and a non-negative function
  $g \in L^q(Q_{\gamma})$ for some $q >1$ and $\bg$ be defined on
  $\R \times \R^d \times \R^d$ from $g$ as in
  Lemma~\ref{l:localization}. Let $s>1$ and consider the union
  $\cG$ of all cylinders $ Q_i\subset Q_\gamma$ such that
  $Q_i= Q_r (z_0)$ with $z_0\in Q_{\gamma}$  and   $r\leq \zeta (z_0)$ and such that,
  \begin{equation}
    \label{hyp des cylindres Qi}
    \fint_{5 Q_i} \bg^q < s^q \leq \fint_{ Q_i} \bg^q.
  \end{equation}
  Then we have: $\{\bg > s\}\cap Q_\gamma \subset \cG$, up to a set of measure
  zero.
\end{lemma}
\begin{proof}
 We notice that both sets $\{\bg > s\}\cap Q_\gamma$ and $\cG$ are contained
  in $Q_\gamma$.  We prove that $\bg(z) \leq s$ for almost all
  $z=(t,x,v)\in Q_\gamma\backslash\cG$.

  We pick $z_r$ in such a way that $z \in Q_r (z_r)$ and  $5Q_r  (z_r)=Q_{5r}(z_{5r})$.
  It is sufficient to consider  $z_r =z \circ (\min(-t,r^2/2),0,0)$.
  Moreover $z_r \in Q_\gamma$ and $Q_{r}(z_r)\subset Q_\gamma$ for $r>0$ small enough.

  Thanks to the Lebesgue differentiation result, Lemma \ref{lem: lebesgue diff}, it is
  sufficient to prove that for  $r >0$ small enough,
  \begin{equation}
    \label{hyp pour Lebesgue2}
    \fint_{Q_r(z_r)} \bg^q < s^q.
  \end{equation}
   We prove
  \eqref{hyp pour Lebesgue2} by contradiction. Lemma \ref{l:zeta_estimate} yields that $\zeta (z_r) \le \zeta (z)+ r/2$, which
  implies that when $r > \zeta(z)$, we have $r>(2/3)\zeta(z_r)$. In particular, \eqref{cond initial normalisation}
  holds which implies that \eqref{hyp pour Lebesgue2} is true for
  $r > \zeta (z)$. The continuity of the function
  $r \mapsto \fint_{\; Q_r(z_r)} \bg^q$ on $(0,+\infty)$ implies
  that if \eqref{hyp pour Lebesgue2} is not true for all $r >0$,
  there is  $\bar r=\max\{r >0: \fint_{\; Q_r(z_r)} \bg^q = s^q \} \le \zeta
  (z) \le \zeta (z_{\bar r}) +\bar r/2$ so that $\bar r \leq 2\zeta (z_{\bar r})$. Thanks to Lemma \ref{l:shifted},
  $5Q_{\bar r} (z_{\bar r})=Q_{5 \bar r}(z_{5 \bar r}) \subset Q_{\gamma} $ which leads to
  \[
    \fint_{\;  5 Q_{\bar r} (z_{\bar r})} \bg^q < s^q = \fint_{\;  Q_{\bar r} (z_{\bar r})} \bg^q.
  \]
  We deduce that the cylinder $Q_{\bar r}(z_{\bar r})$ is
  included in $\cG$. This implies that $z\in \cG$, which
  contradicts $z\in Q_{\gamma}\backslash\cG$.
\end{proof}

\subsection{Reverse H\"older inequality on superlevel sets}

\begin{lemma}[Reverse H\"older on superlevel sets]
  \label{l:reverse-holder-superlevel}
  Given $\gamma>1$ and non-negative functions
  $g \in L^q(Q_{\gamma})$ and $h \in L^\sigma(Q_{\gamma})$, for some
  $\sigma>q >1$ such that there is $b >1$ and $\theta \in (0,1)$ so
  that for all $R>0$ and all $z_0 \in Q_{\gamma}$ such that $Q_{\gamma R}(z_0) \subset Q_{\gamma}$
  the inequality~\eqref{eq:reverse-holder} holds. Let $\bg$ and
  $\bh$ be defined on $\R \times \R^d \times \R^d$ as in
  Lemma~\ref{l:localization}. We consider the superlevel sets
  of $\bg$ and $\bh$:
  \begin{equation*}
    \forall \, t \ge 1, \quad \fg(t):=\{z: \bg(z)>t \} \quad \text{ and }
    \quad \fh(t):=\{z: \bh(z)>t \}.\end{equation*} If
  $\bar \theta \leq \bar \theta_0:= (2.4^q 5^{8d+4})^{-1}$,
  then we have 
  \begin{align*}
    \forall \, t \ge 1, \quad
    \int_{\fg(t)} \bg(z)^q \dd z \le a t^{q-1} \int_{\fg(t)} \bg(z)
    \dd z+a\int_{\fh(t)} \bh(z)^q \dd z,
  \end{align*}
  where $a= 2 \bar b 5^{8d+4} 4^q$, with $\bar b$ appearing in \eqref{eq:holdinv}.
\end{lemma}
\begin{proof}
  In order to prove this inequality, we introduce a parameter $s$
  larger than $t$ such that $s=\alpha t$ with $\alpha>1$ to be
  chosen later.  We split the integral we want to estimate as
  follows
  \begin{equation*}
    \int_{\fg(t)} \bg^q= \int_{\fg(s)} \bg^q+\int_{\fg(t)\setminus
      \fg(s)} \bg^q.
  \end{equation*}
  We estimate the second integral as follows:
  \begin{equation}
    \label{second integ}
    \int_{\fg(t)\setminus \fg(s)} \bg^q\leq s^{q-1} \int_{\fg(t)} \bg
    \leq (\alpha t)^{q-1} \int_{\fg(t)} \bg .
  \end{equation}
  As for the first integral, we use the covering of $\fg(s)$
  obtained in Lemma~\ref{l:covering}:
  $\fg(s)\subset\cG = \cup_i Q_i$. Thanks to
  Lemma~\ref{lem:Vitali} (Vitali's covering argument), we can
  extract a countable family $I$ of disjoint cylinders
  $( Q_i)_{i\in I}$ such that $\cG\subset\cup_{i\in I} 5
  Q_i$. Since $\fg(s)\subset \cup_{i\in I} 5 Q_{i}$ with $Q_i$
  such that $\fint_{5 Q_i} \bg^q < s^q \le \fint_{Q_i} \bg$, we
  have
  \begin{equation}
    \label{eq1}
    \int_{\fg(s)} \bg^q \leq \sum_{i\in I} \int_{5 Q_i} \bg^q\leq
    5^{4d+2} s^q \sum_{i\in I} | Q_i|\leq 5^{4d+2} s^q  |\cG|.
  \end{equation}
  The covering is such that each $Q_i = Q_{r_i} (z_i)$ with
  $r_i \in (0,\zeta (z_i))$, allowing us to apply \eqref{eq:holdinv}:
  \begin{equation*}
    s^q\leq \fint_{ Q_i} \bg^q  \le  \bar b \left\{ \left(
        \fint_{Q_i^{(\gamma)}} \bg
      \right)^q+\fint_{Q_i^{(\gamma)}} \bh^q  \right\}+  \bar
    \theta\fint_{Q_i^{(\gamma)}}  \bg^q 
  \end{equation*}
  where $Q_i^{(\gamma)}$ denotes $Q_{\gamma r_i} (z_i)$, with
  $Q_i =Q_{r_i}(z_i)$.  Using
  $(A+B)^{\frac1q} \le A^{\frac1q}+ B^{\frac1q}$ for all $A,B>0$,
  we get
  \begin{equation*}
    \frac{\alpha t}{\bar b^{1/q}}|Q_i^{(\gamma)}| <
    \int_{Q_i^{(\gamma)}} \bg
    +|Q_i^{(\gamma)}|^{1-\frac{1}{q}}\left( \int_{Q_i^{(\gamma)}}
      \bh^q
    \right)^{1/q}+|Q_i^{(\gamma)}|^{1-\frac{1}{q}}\left(\frac{\bar
        \theta}{\bar b} \int_{Q_i^{(\gamma)}} \bg^q
    \right)^{1/q}.
  \end{equation*}
  We now estimate the three terms on the right hand side of the
  previous inequality. First observe that
  \begin{equation*}
    \int_{Q_i^{(\gamma)}} \bg(z) \dd z = \int_{Q_i^{(\gamma)} \cap \fg(t)} \bg(z) \dd z +
    \int_{Q_i^{(\gamma)} \cap \fg(t)^c} \bg(z) \dd z \le \int_{Q_i^{(\gamma)} \cap \fg(t)}
    \bg(z) \dd z + t|Q_i^{(\gamma)}|,
  \end{equation*}
  As for the second term, we proceed in a similar way:
   \begin{align*}
     |Q_i^{(\gamma)}|^{1-\frac{1}{q}}\left( \int_{Q_i^{(\gamma)}}
     \bh^q \dd z\right)^{1/q}
     &\leq  |Q_i^{(\gamma)}|^{1-\frac{1}{q}}\left(
       \int_{Q_i^{(\gamma)}\cap \fh(t)} \bh^q \dd z + t^q
       |Q_i^{(\gamma)}|\right)^{1/q} \\
     & \leq t|Q_i^{(\gamma)}|\left(
       \frac{1}{|Q_i^{(\gamma)}|}\int_{Q_i^{(\gamma)}\cap \fh(t)}
       \left(\frac{\bh}{t}\right)^q \dd z +  1\right)^{1/q}\\
     & \leq t|Q_i^{(\gamma)}|+t^{1-q}\int_{Q_i^{(\gamma)}\cap
       \fh(t)} \bh^q \dd z
    \end{align*}
    where we have used $(A+1)^{1/q}\leq 1+A$, for any $A\geq 0$
    and $q\ge 1$. We continue similarly for the third term:
    \begin{align*}
      |Q_i^{(\gamma)}|^{1-\frac{1}{q}}\left(\frac{\bar
      \theta}{\bar b} \int_{Q_i^{(\gamma)}} \bg^q \dd
      z\right)^{1/q}  \leq t|Q_i^{(\gamma)}|+\frac{\bar
      \theta}{\bar b}t^{1-q}\int_{Q_i^{(\gamma)}\cap \fg(t)} \bg^q
      \dd z.
    \end{align*}
    Choosing $\alpha = 4 b^{1/q} >1$,  the four last inequalities imply
    \begin{equation*}
      |Q_i^{(\gamma)}| < \frac{1}{t} \left(\int_{Q_i^{(\gamma)}
          \cap \fg(t)} \bg \dd z+t^{1-q}\int_{Q_i^{(\gamma)}\cap
          \fh(t)} \bh^q \dd z+\frac{\bar \theta}{\bar
          b}t^{1-q}\int_{Q_i^{(\gamma)}\cap \fg(t)} \bg^q \dd
        z\right).
    \end{equation*}
    Since $\gamma >1$, $Q_i \subset Q_i^{(\gamma)}$ and $\cG$ is
    covered by the family of cylinders $Q_i^{(\gamma)}$ with
    $Q_i$ as in Lemma~\ref{l:covering}.  Thanks to the Vitali
    covering result Lemma \ref{lem:Vitali}, there is a set
    $J \subset \mathbb{N}$ of indices such that
    $(Q_j^{(\gamma)})_{j\in J}$ are disjoints and
    $\cG \subset \cup_{j\in J}5Q_j^{(\gamma)}$.  Keeping in mind
    that $s = \alpha t$, we deduce from \eqref{eq1} since the
    cylinders $Q_j^{(\gamma)}$ are disjoint,
    \begin{align}
      \label{first integ}
      \nonumber  \int_{G(s)} \bg^q
      & \leq 5^{4d+2} s^q |\cG| \\
      &\leq (\alpha t)^q 5^{8d+4}\sum_{j\in J} |Q_j^{(\gamma)}|
        \nonumber \\
      &\leq (\alpha t)^q 5^{8d+4}\sum_{j\in J} \frac{1}{t}
        \left(\int_{Q_i^{(\gamma)} \cap \fg(t)} \bg \dd
        z+t^{1-q}\int_{Q_i^{(\gamma)}\cap \fh(t)} \bh^q \dd
        z+\frac{\bar \theta}{\bar
        b}t^{1-q}\int_{Q_i^{(\gamma)}\cap \fg(t)} \bg^q \dd
        z\right) \nonumber \\
      &\leq  (\alpha t)^q 5^{8d+4}\frac{1}{t} \left(\int_{\fg(t)}
        \bg \dd z+t^{1-q}\int_{\fh(t)} \bh^q \dd z+\frac{\bar
        \theta}{\bar b}t^{1-q}\int_{\fg(t)} \bg^q \dd z\right).
    \end{align}
    Combining \eqref{second integ} and \eqref{first integ} we
    deduce for
    $\bar \theta \leq \bar \theta_0:= \frac{1}{2.4^q 5^{8d+4}},$
    \begin{equation*}
      \int_{\fg(t)} \bg(z)^q \dd z \le a t^{q-1} \int_{\fg(t)} \bg(z)
      \dd z+a\int_{\fh(t)} \bh(z)^q \dd z
    \end{equation*}
    for any $a \ge 2 \max \{\alpha^{q-1},
    5^{8d+4}\alpha^{q}\}$. Since $\alpha \ge 1$ and
    $\alpha^q = 4^q b$, this condition reduces to
    $a \ge 2\bar b 5^{8d+4} 4^q $.
\end{proof}

\subsection{Reverse  H\"older inequality on level set functions}

\begin{lemma}[Reducing to the monodimensional case]\label{l:reduc-1d}
  Given $\gamma>1$ and non-negative functions
  $g \in L^q(Q_{\gamma})$ and $h \in L^\sigma(Q_{\gamma})$ for some
  $\sigma>q >1$ such that there is $b >1$ and $\theta \in (0,1)$ so
  that for all $R>0$ and all $z_0 \in Q_{\gamma}$ such that $Q_{\gamma R}(z_0) \subset Q_{\gamma}$
  the inequality~\eqref{eq:reverse-holder} holds. Let $\bg$ and
  $\bh$ be defined on $\R \times \R^d \times \R^d$ as in
  Lemma~\ref{l:localization} and the superlevel sets $\fg$ and $\fh$ of $\bg$ and
  $\bh$ as defined in
  Lemma~\ref{l:reverse-holder-superlevel}. The functions
  \begin{equation*}
    \sg(t) := \int_{\fg(t)} \bg(z) \dd z\quad
    \text{ and } \quad  \sh(t)  := \int_{\fh(t)} \bh(z) \dd z
  \end{equation*}
  satisfy
for all $r >1$,
\begin{align}
  \label{lune}
   \fa t \ge 1, \quad -\int_{t} ^{+\infty} s^{r-1}{\rm d}\sg(s)=
  \int_{\fg(t)} \bg(z)^r \dd z, \\
  \label{etlautre}
   \fa t \ge 1, \quad -\int_{t} ^{+\infty} s^{r-1}{\rm d}\sh(s)=
  \int_{\fh(t)} \bh(z)^r \dd z
\end{align}
and
\begin{align}
  \label{eq:ineq-h}
  \fa t \ge 1, \quad - \int_t ^{+\infty}  s^{q-1} \dd \sg(s) \le a t^{q-1} \sg(t) -a \int_t ^{+\infty}  s^{q-1} \dd \sh(s).
\end{align}
\end{lemma}
\begin{proof}
The layer cake representation yields
  \begin{align*}
    \sg(t)
    & = \int_{\R^{2d+1}} \bg(z) \chi_{\{\bg(z) > t\}} \dd z \\
    & = \int_{\R^{2d+1}} \int_0 ^{+\infty} \chi_{\{\bg(z) \chi_{\{\bg(z) >t\}}
      >s\}} \dd s \dd z \\
    & = \int_{\R^{2d+1}} \int_0 ^{t} \chi_{\{\bg(z)
      >s\}}\chi_{\{\bg(z) >t\}} \dd s \dd z
      +\int_{\R^{2d+1}} \int_t ^{+\infty} \chi_{\{\bg(z) >s\}}\chi_{\{\bg(z) >t\}} \dd s
      \dd z \\
    & = \int_{\R^{2d+1}} \int_0 ^{t} \chi_{\{\bg(z) >t\}} \dd s
      \dd z +\int_{\R^{2d+1}} \int_t ^{+\infty} \chi_{\{\bg(z)
      >s\}} \dd s \dd z \\
    & =t|\fg(t)|+ \int_t ^{+\infty} |\fg(s)| \dd s
  \end{align*}
  where we have used the Fubini theorem in the last line.
  This means:
  \begin{equation}
  \label{dh}
   {\rm d}\sg(s)=s {\rm d}|\fg(s)|.
  \end{equation}
  Moreover, we deduce that for any $r \ge 1$:
  \begin{align*}
    u(t):= \int_{\fg(t)} \bg(z)^r \dd z
    & = \int_{\R^{2d+1}} \bg(z)^r \chi_{\{\bg(z) > t\}} \dd z \\
    & = \int_{\R^{2d+1}} \int_0 ^{+\infty} \chi_{\{\bg(z)^r
      \chi_{\{\bg(z) >t\}} >s\}} \dd s \dd z \\
    & = \int_{\R^{2d+1}} \int_0 ^{t^r} \chi_{\{\bg(z)^r
      >s\}}\chi_{\{\bg(z) >t\}} \dd s\\
    & +\int_{\R^{2d+1}} \int_{t^r}^{+\infty} \chi_{\{\bg(z)^r
      >s\}}\chi_{\{\bg(z) >t\}} \dd s \dd z \\
    & = \int_{\R^{2d+1}} \int_0 ^{t^r} \chi_{\{\bg(z) >t\}} \dd
      s+\int_{\R^{2d+1}} \int_{t^r} ^{+\infty} \chi_{\{\bg(z)^r
      >s\}} \dd s\\
    & =t^r|\fg(t)|+ \int_{t^r} ^{+\infty}
      \left|\fg\left(s^{1/r}\right)\right| \dd s
  \end{align*}
  where we have used the Fubini theorem again.  So we have
  \begin{equation}
  \label{df}
   {\rm d}u(s)=s^r {\rm d}|\fg(s)|.
  \end{equation}  
  Combining \eqref{dh} and \eqref{df}, we deduce
  ${\rm d}u(s)=s^{r-1}{\rm d}\sg(s)$, which implies
  \eqref{lune}. Identity \eqref{etlautre} is obtained by
  replacing $\fg(t)$ and $\bg(z)$ with $\fh(t)$ and $\bh
  (z)$. Finally~\eqref{eq:ineq-h} follows from
  Lemma~\ref{l:reverse-holder-superlevel}.
\end{proof}

\subsection{Proof of the improved integrability in dimension one}

\begin{lemma}[The dimension one case] \label{l:1d case} The
  functions $\sg$ and $\sh$ defined in Lemma~\ref{l:reduc-1d}
  satisfy,
  \begin{multline*}
    - \int_1 ^{+\infty} t^{p-1} \dd \sg(t) \\
    \le   \left( \frac{ (p-1)- (p-q) }{ (p-1) - a
    (p-q) } \right)\left( - \int_1 ^{+\infty}    t^{q-1} \dd
  \sg(t) \right)+ \frac{2a(p-1)}{(p-1) - a (p-q)} \left( -\int_1
  ^{+\infty} t^{p-1} \dd \sh(t) \right)
\end{multline*}
  for $p \in [q,p^*)$ with $p^* =\min\left(\sigma,q+\frac{q-1}{a-1}\right)$.
  We recall that $a = 2 (75)^{4d+2} 4^q b >1$.
\end{lemma}
\begin{proof}
 Consider for  $r \ge 1$,
  \begin{equation*}
    I(r) := - \int_1 ^{+\infty} t^{r-1} \dd \sg(t) \quad \text{ and
    } \quad  J(r) := - \int_1 ^{+\infty} t^{r-1} \dd \sh(t).
  \end{equation*}
  Pick $p \in [q,p^*)$ and consider for $L>1$,
  \begin{equation*}
    I_L (p ) =- \int_1 ^L t^{p-q} t^{q-1} \dd \sg(t).
  \end{equation*}
  We perform an integration by parts,
  \begin{align*}
    I_L(p)
    & = - \int_1 ^L t^{p-q} t^{q-1} \dd \sg(t) \\
    & = \left[ t^{p-q} \left(\int_t ^{+\infty} s^{q-1} \dd
      \sg(s)\right) \right]_1 ^L + (p-q) \int_1 ^L t^{p-q-1} \left(
      - \int_t ^{+\infty} s^{q-1} \dd \sg(s) \right) \dd t \\
    & = L^{p-q} \int_L^{+\infty} s^{q-1} \dd \sg (s) 
      + I(q) + (p-q) \int_1 ^L t^{p-q-1} \left( - \int_t ^{+\infty}
      s^{q-1} \dd \sg(s) \right) \dd t.
  \end{align*}
  We now use~\eqref{eq:ineq-h} and get,
  \begin{multline}
    \label{ineq I1}
    I_L(p) \le L^{p-q} \int_L^{+\infty} s^{q-1} \dd \sg (s)
    + I(q) + a (p-q) \int_1 ^L t^{p-2} \sg(t) \dd t \\
    + a (p-q) \int_1 ^{L} t^{p-q-1} \left( - \int_t^{+\infty}
      s^{q-1} \dd\sh(s)  \right) \dd t.
  \end{multline}
  First, we perform an integration by parts in the third term of
  the right hand side of~\eqref{ineq I1}:
  \begin{align*}
    a (p-q) \int_1 ^L t^{p-2} \sg(t) \dd t
    & =  a (p-q) \int_1 ^L \left(\frac{t^{p-1} -1}{p-1} \right)'
      \sg(t) \dd t \\
    & = a (p-q) \left\{ \left[ \left(\frac{t^{p-1} -1}{p-1}
      \right) \sg(t) \right]_1 ^{L} - \frac{1}{p-1} \int_1 ^L
      \left( t^{p-1} - 1 \right) \dd \sg(t) \right\} \\
    & = a (p-q) \left\{\frac{L^{p-1} -1}{p-1} \sg (L)+
      \frac{I_L(p)}{p-1} + \frac{1}{p-1} \int_1 ^L \dd \sg(t)
      \right\} \\
    &  = a  (p-q) \left\{\frac{L^{p-1}}{p-1} \sg (L) +
      \frac{I_L(p)}{p-1} - \frac{1}{p-1} \sg(1)\right\}.
  \end{align*}
  Second, inequality~\eqref{eq:ineq-h} at $t=1$ yields
  \begin{equation*}
    - \sg(1) \le \frac{1}{a} \int_1 ^\infty  s^{q-1} \dd \sg(s) -
    \int_1 ^\infty  s^{q-1} \dd \sh(s) = -    \frac{1}{a} I(q)
    +J(q).
  \end{equation*}
  We thus proved the following estimate,
  \begin{equation}
  \label{ineq I2}
  a (p-q) \int_1 ^L  t^{p-2} \sg(t) \dd t   \le  a (p-q)
  \left\{\frac{L^{p-1}}{p-1} \sg (L) + \frac{I_L(p)}{p-1} -
    \frac{1}{a(p-1)} I(q) + \frac1{p-1} J(q) \right\}.
  \end{equation}
  The fourth term of the right hand side of~\eqref{ineq I1} is
  bounded by
  \begin{align}
    \label{ineq I3}
    -a (p-q) \int_1 ^L t^{p-q-1} \int_t^{+\infty} s^{q-1} \dd \sh(s) \dd t
    & \le -a (p-q) \int_1 ^{+\infty} \int_1^{+\infty} t^{p-q-1}
      s^{q-1} \mathbbm{1}_{s\ge t} \dd \sh(s) \dd t\nonumber\\
    &=  -a (p-q) \int_1 ^{+\infty} \left\{ \int_1^{s} t^{p-q-1}
      \dd t \right\} s^{q-1} \dd \sh(s)\nonumber\\
    \nonumber  &= -a \int_1 ^{+\infty} s^{p-1}  \dd \sh(s) \\
    & = a J (p).
  \end{align}
 
  Combining \eqref{ineq I1}, \eqref{ineq I2} and \eqref{ineq I3}
  finally yields,
  \begin{align*}
    I_L(p)  \le
    & I(q) + a (p-q) \left\{  \frac{I_L(p)}{p-1} -
      \frac{1}{a(p-1)} I(q) +\frac1{p-1} J (q) \right\}+a J(p) \\
    + & L^{p-q} \left[ \int_L^{+\infty} s^{q-1} \dd \sg (s) +
        \frac{a(p-q)}{p-1} L^{q-1} \sg (L)\right].
  \end{align*}
  Observe that for $p \in [q,p^*)$ with
  $p^* =\min\left(\sigma,\frac{aq-1}{a-1}\right)$, we have $\frac{a(p-q)}{p-1} \le 1$
  and
  \begin{equation*}
    \left[\int_L^{+\infty} s^{q-1} \dd \sg (s) +
      \frac{a(p-q)}{p-1} L^{q-1} \sg (L) \right] \le
    \int_L^{+\infty} s^{q-1} \dd \sg (s) + L^{q-1} \sg (L) \le 0.
  \end{equation*}
  We thus proved that for all $L >1$,
  \begin{equation*}
    I_L(p) \le \left( \frac{ (p-1)- (p-q)}{ (p-1) - a (p-q) }
    \right) I(q) + \frac{a (p-q)}{(p-1)-a (p-q)} J(q)+
    \frac{a(p-1)}{(p-1) - a (p-q)} J(p)
  \end{equation*}
  Remark that $J(q) \le J(p)$ since $q \le p$. Letting
  $L \to +\infty$ allows us to conclude the proof.
\end{proof}

Recalling the definition of  $\sg$ and $\sh$, see
Lemma~\ref{l:reduc-1d}, and formulas~\eqref{lune} and
\eqref{etlautre}, Lemma~\ref{l:1d case} can be reformulated as
follows.
\begin{lemma}[Going back to $\bg$ and $\bh$]\label{l:back-to-g}
For $p \in [q,p^*)$,
  \begin{equation*}
    \left( \int_{\fg(1)} \bg^p \dd x \right)^{1/p} \le C_{p,q,a}
    \left( \left( \int_{\fg(1)} \bg^q \dd z \right)^{1/q}+ \left(
        \int_{ \fh(1)} \bh^p \dd z \right)^{1/p} \right)
  \end{equation*}
  with $C_{p,q,b} = \frac{2a (p-1)}{(p-1)-a(p-q)}$ with $a = 2 (75)^{4d+2} 4^q b$.
\end{lemma}


\subsection{Proof of Theorem~\ref{thm:gehring}}
\begin{proof}[Proof of Theorem~\ref{thm:gehring}]
  The localization function $\zeta$ is bounded from below on
  $Q_1$ by $c_\gamma>0$ with 
  \begin{equation}
    \label{e:cgamma}
    c_\gamma =\frac{1}{2} \min \left (\frac{\gamma-1}5,
      \left(\frac{\gamma^2-1}{13}\right)^{\frac12},
      \left(
        \frac{\gamma^3-1}{25\gamma}\right)^{\frac12}
    \right).
  \end{equation}
  With such a lower bound at hand, we now write,
  \begin{align*}
    \left( \int_{Q_1} g^p \right)^{\frac1p}
    = &C_0 \|g\|_{L^q (Q_\gamma)} \left(\int_{Q_1}
        |Q_{\zeta(z)}|^{-p/q} \bg^p \right)^{\frac1p} \\
    \le& C_1 C_0 \|g\|_{L^q (Q_\gamma)}  \left(\int_{Q_\gamma}  \bg^p \right)^{\frac1p} \\
    \le& C_1 C_0 \|g\|_{L^q (Q_\gamma)}  \left(\int_{\fg (1)}  \bg^p
         \right)^{\frac1p} + C_1 |Q_\gamma|^{\frac1p} C_0
         \|g\|_{L^q (Q_\gamma)} 
  \end{align*}
  with $C_1= |Q_1|^{-1/q} c_\gamma^{-\frac{4d+2}q}$.  From
  Lemma~\ref{l:back-to-g}, we have
  \begin{align*}
    C_0 \|g\|_{L^q (Q_\gamma)}  \left(\int_{\fg (1)}  \bg^p
    \right)^{\frac1p}  \le C_{p,q,b} \left(
    |Q_1|\|\zeta\|_\infty^{\frac{4d+2}q} \| g\|_{L^q (Q_\gamma)}
    + |Q_1|\|\zeta\|_\infty^{\frac{4d+2}q} \|h\|_{L^p(Q_\gamma)}\right).
  \end{align*}
  Recall now that $\zeta$ is bounded by $\gamma/5$ and conclude
  by combining the two last inequalities with
  \begin{equation*}
    C_G = C_1 C_0 |Q_\gamma|^{\frac1p} + C_1
    C_{p,q,b}|Q_1| (\gamma/5)^{\frac{4d+2}q}.
  \end{equation*} 
  We recall that $C_0 = 2^{\frac{4d+2}q}$ (see  Lemma~\ref{l:localization}) and
  $C_1 = |Q_1|^{-1/q} c_\gamma^{-\frac{4d+2}q}$ with $c_\gamma$
  given by \eqref{e:cgamma}.
\end{proof}

\section{Application to kinetic equations} 
\label{sec:Gain gradient}

\subsection{A zeroth order Poincaré-Wirtinger inequality}

The proof of the reverse H\"older inequality for the gradient of
solutions to the kinetic Fokker-Planck equation makes use of a
\emph{local hypoelliptic} Poincaré-Wirtinger inequality. The
proof of this inequality relies on the following zeroth order
Poincaré-Wirtinger inequality.
\begin{lemma}[Zeroth order Poincaré-Wirtinger inequality]
  \label{l:0PW}
  Given a function $g =g(t,x)\in L^q(\cQ)$ with $q\in (1,+\infty)$ in some convex open
  set $\cQ \subset \R^{1+d}$, then
  \begin{equation*}
    \left\| g -  \fint_{\cQ} g \right\|_{L^q (\cQ)} 
    \le C_{PW} ^0 \left\| \nabla_{t,x} g \right\|_{W^{-1,q}(\cQ)} 
\end{equation*}
  for some constant $C_{PW} ^0>0$ depending on $\cQ$.
\end{lemma}
In order to establish such an inequality, we first solve the
divergence equation for some $L^p$ function $G$ with zero mean on
$\cQ$. The solution is not unique, but it is crucial to our
argument, and it is a remarkable fact, that it is possible
construct a solution that vanishes at the boundary of $\cQ$. Such
constructions are detailed in \cite{MR2808162}
and go back to~\cite{MR553920}. The following statement is
extracted from~\cite[Lemma~III.3.1,~p.~162]{MR2808162}.
\begin{proposition}[The divergence equation]\label{prop:h}
  Let $p \in (1,+\infty)$ and let $\cQ$ be a cylinder of the form
  $(-r^2,0] \times B_{r^3} \subset \R \times \R^d$.  Then there
  is $\Cdiv >0$ so that for any $G \in L^p(\cQ)$
  with $\int_{\cQ} G(t,x) \dd t \dd x =0$, there exists a
  vector-valued solution $\vec{h} \in W^{1,p}_0(\cQ,\R^{1+d})$ (the
  standard Sobolev space with Dirichlet conditions) to 
  \begin{equation*}
    \nabla_{t,x} \cdot \vec{h} = G \ \mbox{ on } \ \cQ
  \end{equation*}
  which satisfies
  \begin{equation*}
    \|h\|_{W^{1,p}_0(\cQ)} \le \Cdiv \| G \|_{L^p(\cQ)}.
  \end{equation*}
\end{proposition}
With such a proposition at hand, we can prove the previous lemma.
\begin{proof}[Proof of Lemma~\ref{l:0PW}]
  For any real number $a \in \R$, we write
  $[a]^{q-1} := |a|^{q-2}a$ if $a \neq 0$ and $[0]^{q-1}=0$. In
  particular, $|a|^q = [a]^{q-1} a$. We then consider the vector
  field $\vec{h}$ given by Proposition~\ref{prop:h} for
  \begin{equation*}
    G:= \left[g-  \fint_{\cQ} g\right]^{q-1}- \fint_{\cQ}\left[g-
      \fint_{\cQ} g\right]^{q-1}.
  \end{equation*}
  Let us check that $G\in L^p (\cQ)$ where $p$ satisfies $\frac{1}{p}+\frac{1}{q}=1$. We write $G$ as
  $G_0 - \fint_\cQ G_0$ with $G_0 = \left[g- \fint_{\cQ} g\right]^{q-1}$.  Then
  \begin{equation*}
    \|G_0\|_{L^p (\cQ)}^p \le \int_\cQ \left|g - \fint_\cQ g
    \right|^{p (q-1)} =\int_\cQ \left|g - \fint_\cQ g
    \right|^{q}.
  \end{equation*}
  This means that
  \begin{equation*}
    \|G_0\|_{L^p (\cQ)} \le \left\| g - \fint_\cQ
      g \right\|_{L^q(\cQ)}^{q-1}.
  \end{equation*}
  Moreover, Jensen's inequality yields,
  \begin{equation*}
    \left\| \fint_\cQ G_0 \right\|_{L^p(Q)} \le \left\| g -
      \fint_\cQ g \right\|_{L^q(\cQ)}^{q-1}
  \end{equation*}
  so that
  \begin{equation*}
    \|G\|_{L^p (\cQ)} \le 2 \left\| g - \fint_\cQ
      g \right\|_{L^q(\cQ)}^{q-1}.
  \end{equation*}

  We now compute
  \begin{align*}
    \int_{\cQ} \left| g(t,x) - \fint_{\cQ} g \right|^q \dd t \dd x
    & = \int_{\cQ} G \left(g -  \fint_{\cQ} g \right)\\
    &= \int_{\cQ}  \left(g(t,x)-\fint_{\cQ} g \right) \nabla_{t,x}\cdot h(t,x) \dd t \dd x \\
    &= - \int_{\cQ} \nabla_{t,x} g(t,x) \cdot h(t,x) \dd t \dd x \\
    &\le  \left\| \nabla_{t,x} g(t,x) \right\|_{W^{-1,q}(\cQ)}
      \left\|h \right\|_{W^{1,p}_0(\cQ)} \\
    & \le \Cdiv \left\| \nabla_{t,x} g(t,x)
      \right\|_{W^{-1,q}(\cQ)} \left\|G \right\|_{L^p(\cQ)}.
  \end{align*}
  The estimate of the $L^p$ norm of $G$ concludes the proof with
  $C_{PW}^0 = 2 \Cdiv$.
\end{proof}

\subsection{Poincar\'e inequality in $L^{q}$}
\label{sec:Poincare}

We prove here a \emph{local hypoelliptic} Poincar\'e-Wirtinger
inequality for solutions to~\eqref{e:main}, following ideas of
\cite{am19}.
\begin{theorem}[Local hypoelliptic Poincar\'e-Wirtinger
  inequality] \label{t:hpw} Let $Q =Q_r$ be a cylinder of
  $\R \times \R^d \times \R^d$ for some $r>0$.  For any function
  $f \in L^2 (Q)$ and any $1<q\leq 2$, we have,
  \begin{equation}
    \label{eq:pw-pure}
    \left\| f - \langle \langle f  \rangle \rangle_{Q} \right\|_{L^q (Q)}
     \le C_{hpw} \left( \left\| (\partial_t + v\cdot \nabla_x)
         f\right\|_{L^q(\cQ;W^{-1,q}(B))} +
       \left\| \nabla_v f \right\|_{L^q (Q)} \right)
   \end{equation}
   where $\cQ = (-r^2,0]\times B_{r^3}$ and $B=B_r$ and the
   average is defined as
   \begin{equation*}
     \langle \langle f \rangle \rangle_{Q} :=  \fint_{Q} f(t,x,v)
     \dd t \dd x \dd v.
   \end{equation*}
   The constant $C_{hpw} >0$ only depends on $Q$.
\end{theorem}
\begin{proof}
  The proof is a variation along the ideas
  in~\cite[Theorem~7.2-(1)]{am19}, where we replace the Gaussian
  Poincar\'e inequality with the Poincar\'e-Wirtinger inequality
  in $v$:
  \begin{equation}
    \label{eq:pw}
    \int_{B} \left|g - \langle g \rangle_{B} \right|^q \dd v  \le
    C_{PW} \int_{B} \left| \nabla_v g \right|^q \dd v
    \quad \mbox{ with } \quad \langle g \rangle_{B} :=
    \frac{1}{|B|} \int_{B} g(v) \dd v.
  \end{equation}

  We use the partial $v$-average $\langle f \rangle_{B}(t,x)$ as an
  intermediate term in the inequality:
  \begin{equation*}
    \int_{Q} \left|f - \langle\langle f \rangle\rangle_{Q} \right|^q
    \dd t \dd x \dd v 
    \le 2\int_{Q} \left|f - \langle f \rangle_{B} \right|^q
      \dd t \dd x \dd v + 
     2 \int_{Q} \left|\langle f \rangle_{B}
      - \langle\langle f \rangle\rangle_{Q} \right|^q \dd t \dd x \dd v.
  \end{equation*}
  
  The first term in the right hand side is estimated with \eqref{eq:pw}
  applied at each $(t,x) \in \cQ$:
  \begin{align*}
    \int_{Q} \left|f - \langle f \rangle_{B} \right|^q
      \dd t \dd x \dd v \le C_{PW} \int_{Q} \left| \nabla_v f
    \right|^q \dd t \dd x \dd v. 
  \end{align*}
  
  To estimate the second term we first remove the $v$-integration
  and apply Lemma~\ref{l:0PW}:
  \begin{align}\nonumber
    \int_{\cQ \times B} \left|\langle f \rangle_{B}
    - \langle\langle f \rangle\rangle_{\cQ \times B} \right|^q
    \dd t \dd x \dd v
    &= |B| \int_{\cQ} \left|\langle f \rangle_{B}
      - \langle\langle f \rangle\rangle_{\cQ \times B} \right|^q
      \dd t \dd x\\
    &\le |B| C_{PW} ^0 \left\| \nabla_{t,x} \langle f \rangle_{B}
      \right\|_{W^{-1,q}(\cQ)} ^q. 
  \end{align}
  
  We are left with proving the following hypoelliptic estimate
  \begin{equation}
    \label{eq:hypo}
    \left\| \nabla_{t,x} \langle f \rangle_{B}
      \right\|_{W^{-1,q}(\cQ)} ^2 \le C_{hypo} \left( \left\| \nabla_v f
    \right\|_{L^q(\cQ\times B)} ^2 + \left\| \left(\partial_t + v\cdot
    \nabla_x\right) f\right\|_{L^q(\cQ;W^{-1,q}(B))} ^2 \right)
  \end{equation}
  for some constant $C_{hypo} >0$ depending on $\cQ$. Let us
  denote $v_0:=1$, and $\partial_0:=\partial_t$, $\partial_i :=
  \partial_{x_i}$ for $i=1,\dots,d$. Consider
  $\varphi(t,x) \in W^{1,p}_0(\cQ)$ and $\chi_i \in C^\infty_c(B)$,
  $i=0,1,\dots,d$ such that
  \begin{equation*}
    \fa i,j=0,1,\dots,d, \quad
     \int_{B} \chi_i(v) v_j \dd v =\delta_{ij}.
  \end{equation*}
  We then integrate any derivative
  $\partial_i \langle f \rangle_{B}$ for $i=0,1,\dots,d$, against
  $\varphi$ and compute:
  \begin{align*}
    & \int_{\cQ} \partial_i \langle f \rangle_{B}(t,x)
      \varphi(t,x) \dd t  \dd x\\
    & = \int_{\cQ \times B} \partial_i \langle f \rangle_{B}(t,x)
      \varphi(t,x) v_i \chi_i(v) \dd t \dd x \dd v \\
    & = \int_{\cQ \times B} \left(\partial_t + v \cdot
      \nabla_x\right) \langle f \rangle_{B}(t,x) \varphi(t,x)
      \chi_i(v) \dd t
      \dd x \dd v \\
    & = \int_{\cQ \times B} \left(\partial_t + v \cdot \nabla_x
      \right) \left( \langle f \rangle_{B}(t,x) -f(t,x,v) \right)
      \varphi(t,x) \chi_i(v) \dd t
      \dd x \dd v \\
    & \qquad \qquad + \int_{\cQ \times B} \left(\partial_t + v
      \cdot \nabla_x \right) f(t,x,v) \varphi(t,x) \chi_i(v) \dd t
      \dd x \dd v  =: I_1 ^i + I_2 ^i.
  \end{align*}
  The first term is controlled from the Poincar\'e-Wirtinger
  inequality in $v$ again:
  \begin{align*}
    \left|I_1^i \right| \le \norm{f-\langle f \rangle_{B}}_{L^q(\cQ\times B)}
    \norm{\left(\partial_t + v \cdot \nabla_x \right)(\varphi
    \chi_i)}_{L^p(\cQ\times B)} \lesssim \norm{\nabla_v
    f}_{L^q(\cQ \times B)}. 
  \end{align*}
  The second term is controlled by
  \begin{align*}
    \left|I_2^i\right| \le \norm{\left(\partial_t + v \cdot
    \nabla_x\right)f}_{L^q(\cQ;W^{-1,q}(B))}
    \norm{\varphi}_{L^p(\cQ)} \norm{\chi_i}_{W^{1,p}_0(B)}.
  \end{align*}
  This establishes the hypoelliptic estimate~\eqref{eq:hypo} and
  concludes the proof of~\eqref{eq:pw-pure}.
\end{proof}
\begin{corollary}[Local $L^q$ estimate]\label{c:hpw}
  For any function $f \in L^2 (Q)$ solution to~\eqref{e:main} on
  $Q = \cQ \times B$ we have for any $1<q\leq 2$,
  \begin{equation}
    \label{eq:pw-applied}
    \left\| f - \langle\langle f  \rangle\rangle_{Q}
    \right\|_{L^q (Q)} \le C_{hpw} ' \left( \left\| \nabla_v f
      \right\|_{L^q (Q)} + \norm{S}_{L^2 (Q)}\right).
  \end{equation}
  The constant $C'_{hpw} >0$ only depends on $Q$ and the
  ellipticity constant $\Lambda$.
\end{corollary}
\begin{proof}
Let us now use the equation \eqref{e:main} to prove that
  \begin{align*}
    \norm{ (\partial_t + v\cdot    \nabla_x)
    f}_{L^q(\cQ;W^{-1,q}(B))}  \lesssim \norm{\nabla_v
    f}_{L^q(\cQ\times B)}  + \norm{S}_{L^q (Q)}.
  \end{align*}
  Consider $\varphi \in L^p(\cQ;W^{1,p}_0(B))$ and compute
  \begin{align*}
    \left| \int_{Q} ( \partial_t  + v\cdot \nabla_x ) f \varphi
    \dd z \right|
    & \le  \left| \int_{Q} \left( \cA \nabla_v f \right) \cdot
      \nabla_v  \varphi \dd t \dd x \dd v \right| 
      + \left| \int_{Q} (\cB \cdot \nabla_v f)    \varphi \dd t \dd
      x \dd v \right| + \left| \int_{Q} \cS \varphi \dd t \dd x
      \dd v \right| \\
    & \le \Lambda \left( \norm{\nabla_v f}_{L^q(Q)} +
      \norm{\cS}_{L^q (Q)} \right)
      \norm{\varphi}_{L^p(\cQ;W^{1,p}_0(B))}
  \end{align*}
  which concludes the proof.  
\end{proof}
\begin{corollary}[Gain of integrability of the solution controlled by the gradient]
  \label{cor:gain}
  Let $f$ be a weak solution to~\eqref{e:main} in $Q$ (as per
  Definition~\ref{defi:weak-sol}). Then for all cylinders
  $Q_r(z_0)$ such that $Q_{2r}(z_{0})\subset Q$, $f$ satisfies
   \begin{equation*}
     \norm{f-\langle\langle f
       \rangle\rangle_{Q_r(z_0)}}_{L^{p}(Q_r(z_0))}\le C_r \left(
       \norm{\nabla_v f}_{L^{2}(Q_{2r}(z_0))} +
       \norm{S}_{L^2(Q_{2r}(z_0))}\right),
   \end{equation*}
   where $p=\frac{6(2d+1)}{6d+1}>2$ and $C_r >0$ only depends on
   the dimension $r$, $d$, $\lambda$ and $\Lambda$.
\end{corollary}
\begin{proof}
  The result follows from the combination of Proposition
  \ref{prop:Gain sol} and Corollary \ref{c:hpw} with $q=2$ applied
  to $f-\langle\langle f \rangle\rangle_{Q_r(z_0)}$, which is a
  weak solution to~\eqref{e:main}.
\end{proof}

\subsection{Proof of the gain of integrability of the velocity
  gradient}

\begin{proof}[Proof of Theorem \ref{t:gain velocity gradient}]
  We first prove that the velocity gradient of any solution
  $\overline{f}$ to~\eqref{e:main} satisfies
  \eqref{eq:reverse-holder} for the cylinders $Q_1$ and
  $Q_{\gamma}$ and for $q=2$. 
  We use Proposition \ref{prop:EE},an interpolation inequality of $L^2$
  between $L^p$ and $L^{p'}$ where $p>2$ is defined in
  Corollary~\ref{cor:gain} and $p':=p/(p-1) \in (1,2)$ and a Young inequality:
  \begin{align*}
    \norm{\nabla_v \overline{f}}_{L^{2}(Q_1)}^2
    &\le C \norm{\overline{f}-\langle\langle
      \overline{f}\rangle\rangle_{Q_{(\gamma+1)/2}}}_{L^{2}(Q_{(\gamma+1)/2})}^2
      + C\norm{\overline{\cS}}_{L^2 (Q_\gamma)}^2   \\
      &\le \norm{\overline{f}-\langle\langle
      \overline{f}\rangle\rangle_{Q_{(\gamma+1)/2}}}_{L^{p}(Q_{(\gamma+1)/2})}\norm{\overline{f}-\langle\langle
      \overline{f}\rangle\rangle_{Q_{(\gamma+1)/2}}}_{L^{p'}(Q_{(\gamma+1)/2})}+ C\norm{\overline{\cS}}_{L^2 (Q_\gamma)}^2 \\
    &\le \eps \norm{\overline{f}-\langle\langle
      \overline{f}\rangle\rangle_{Q_{(\gamma+1)/2}}}_{L^{p}(Q_{(\gamma+1)/2})}^{2}
      + C\norm{\overline{f}-\langle\langle
      \overline{f}\rangle\rangle_{Q_{(\gamma+1)/2}}}_{L^{p'}(Q_{(\gamma+1)/2})}^{2}
      + C\norm{\overline{\cS}}_{L^2 (Q_\gamma)}^2 \\
    \intertext{we now use Corollary~\ref{cor:gain} on the first
    term and Corollary~\ref{c:hpw} with $q=p'$ on the
    second term:}
    & \le \eps C \norm{\nabla_v \overline{f}}_{L^{2}(Q_{\gamma})}^2
      + C \norm{\nabla_v \overline{f}}_{L^{p'}(Q_{(\gamma+1)/2})}^2
      + C\norm{\overline{\cS}}_{L^2 (Q_\gamma)}^2 \\
          & \le \eps C \norm{\nabla_v \overline{f}}_{L^{2}(Q_{\gamma})}^2
      + C' \norm{\nabla_v \overline{f}}_{L^2(Q_{(\gamma+1)/2})}^{2\delta}\norm{\nabla_v \overline{f}}_{L^1(Q_{(\gamma+1)/2})}^{2(1-\delta)}
      + C\norm{\overline{\cS}}_{L^2 (Q_\gamma)}^2 \\
    & \le 2 \eps C \norm{\nabla_v \overline{f}}_{L^{2}(Q_{\gamma})}^2
      + C'' \norm{\nabla_v \overline{f}}_{L^1(Q_{(\gamma+1)/2})}^2
      + C\norm{\overline{\cS}}_{L^2 (Q_\gamma)}^2
  \end{align*}
  where we have used an interpolation inequality of $L^{p'}$ betwenn $L^1$ and $L^2$ with $\delta\in (0,1)$ such that $\frac{1}{p'}=\frac{\delta}{2}+\frac{1-\delta}{1}$ and a Young inequality in the last
  line. From the previous inequality with $\theta=2 \varepsilon C$
  with $\varepsilon$ small enough so that $\theta<\theta_0$
  follows
  \begin{align}
    \label{hyp Gehring cylindre fixe}
    \fint_{Q_{1}} |\nabla_v \overline{f}|^2 \le
    \theta \fint_{Q_{\gamma}} |\nabla_v \overline{f}|^2+
    C(\theta) \left(\fint_{Q_{\gamma}} |\nabla_v
    \overline{f}| \right)^{2} + C(\theta) \fint_{Q_\gamma}
    \overline{\cS}^2.
  \end{align}
  Now let us prove that the velocity gradient of any weak
  solution $f$ to~\eqref{e:main} satisfies
  \eqref{eq:reverse-holder} for any cylinder such that
  $Q_{\gamma r}(z_0) \subset Q_{\gamma R}(z)$ with
  $z_0=(t_0,x_0,v_0)\in Q$. Consider such a solution $f$ and a
  cylinder $Q_{r}(z_0)$ such that
  $Q_{\gamma r}(z_0) \subset Q_{\gamma R}(z)$.  We apply
  \eqref{hyp Gehring cylindre fixe} to the solution 
  $\overline{f}(t,x,v)=f(r^2t+t_0,r^3x+x_0+r^2tv_0,rv+v_0)$ of the equation~$\eqref{e:main}$ with different coefficients $\bar{\cB}=r\cB, \bar{\cS}=r^2\cS$ satisfying the same conditions, so
  that
  \begin{equation*}
    f(t,x,v)=
    \overline{f}\left(\frac{t-t_0}{r^2},\frac{x-x_0-(t-t_0)v_0}{r^3},
    \frac{v-v_0}{r}\right).
  \end{equation*}
  This yields
  \begin{align*}
    & \fint_{Q_r(z_0)} |\nabla_v f|^2=
    \frac{1}{r^2}\fint_{Q_{1}} |\nabla_v \overline{f}|^2,\\
    & \fint_{Q_{\gamma r}(z_0)} |\nabla_v f|=
    \frac{1}{r}\fint_{Q_{\gamma}} |\nabla_v
    \overline{f}|,\\
    & \fint_{Q_{\gamma r}(z_0)} |\nabla_v f|^2=
    \frac{1}{r^2}\fint_{Q_{\gamma}} |\nabla_v \overline{f}|^2,\\
     & \fint_{Q_{\gamma r}(z_0)} \cS^2=
    \frac{1}{r^4}\fint_{Q_{\gamma}} \overline{\cS}^2.
  \end{align*}
  Combining those inequalities with~\eqref{hyp Gehring cylindre
    fixe} and the fact that $r\leq 1$ gives
  \begin{equation*}
    \fint_{Q_r(z_0)} |\nabla_v f|^2 \le \theta \fint_{Q_{\gamma
        r}(z_0)} |\nabla_v f|^2
    + C(\theta) \left[ \left(\fint_{Q_{\gamma r}(z_0)} |\nabla_v
        f|\right)^{2}+ \fint_{Q_{\gamma r}(z_0)} \cS^{2}
    \right],
  \end{equation*}
  which is \eqref{eq:reverse-holder} with $g = |\nabla_v f|$
  and $q=2$ and $b = C(\theta)$ and $h =\cS$.  We now apply
  Theorem~\ref{thm:gehring} and get the desired gain of
  integrability for $\nabla_v f$.
\end{proof}

\bibliographystyle{amsalpha}
\bibliography{Gehring}

\providecommand{\bysame}{\leavevmode\hbox to3em{\hrulefill}\thinspace}
\providecommand{\MR}{\relax\ifhmode\unskip\space\fi MR }
\providecommand{\MRhref}[2]{%
  \href{http://www.ams.org/mathscinet-getitem?mr=#1}{#2}
}
\providecommand{\href}[2]{#2}
\begin{thebibliography}{DHHM23}

\bibitem[AAMN21]{am19}
D.~Albritton, S.~Armstrong, J.~C. Mourrat, and M.~Novack, \emph{Variational
  methods for the kinetic fokker-planck equation}, 2021.

\bibitem[ABES20]{MR4167265}
Pascal Auscher, Simon Bortz, Moritz Egert, and Olli Saari, \emph{Non-local
  {G}ehring lemmas in spaces of homogeneous type and applications}, J. Geom.
  Anal. \textbf{30} (2020), no.~4, 3760--3805. \MR{4167265}

\bibitem[AIN24]{auscher:hal-04519638}
Pascal Auscher, Cyril Imbert, and Lukas Niebel, \emph{{Weak solutions to
  Kolmogorov-Fokker-Planck equations: regularity, existence and uniqueness.}},
  40 pages, submitted., March 2024.

\bibitem[AL99]{MR1754355}
A.~A. Arkhipova and O.~A. Ladyzhenskaya, \emph{On a generalization of
  {G}ehring's lemma}, Zap. Nauchn. Sem. S.-Peterburg. Otdel. Mat. Inst.
  Steklov. (POMI) \textbf{259} (1999), 17--18, 296. \MR{1754355}

\bibitem[Ark97]{zbMATH01285988}
A.~A. Arkhipova, \emph{On modifications of the {Gehring} {Lemma} appearing in
  the study of parabolic initial-boundary value problems}, Probl. Mat. Anal.
  \textbf{17} (1997), 20--45 (Russian).

\bibitem[Bog79]{MR553920}
M.~E. Bogovski\u{\i}, \emph{Solution of the first boundary value problem for an
  equation of continuity of an incompressible medium}, Dokl. Akad. Nauk SSSR
  \textbf{248} (1979), no.~5, 1037--1040. \MR{553920}

\bibitem[DHHM23]{dietert2023quantitative}
Helge Dietert, Frédéric Hérau, Harsha Hutridurga, and Clément Mouhot,
  \emph{Quantitative geometric control in linear kinetic theory}, 2023.

\bibitem[FSC96]{MR1398170}
Bruno Franchi and Francesco Serra~Cassano, \emph{Gehring's lemma for metrics
  and higher integrability of the gradient for minimizers of noncoercive
  variational functionals}, Studia Math. \textbf{120} (1996), no.~1, 1--22.
  \MR{1398170}

\bibitem[Gal11]{MR2808162}
G.~P. Galdi, \emph{An introduction to the mathematical theory of the
  {N}avier-{S}tokes equations}, second ed., Springer Monographs in Mathematics,
  Springer, New York, 2011, Steady-state problems. \MR{2808162}

\bibitem[Geh73]{MR0402038}
F.~W. Gehring, \emph{The {$L^{p}$}-integrability of the partial derivatives of
  a quasiconformal mapping}, Acta Math. \textbf{130} (1973), 265--277.
  \MR{0402038}

\bibitem[GI23]{zbMATH07750909}
Jessica Guerand and Cyril Imbert, \emph{Log-transform and the weak {Harnack}
  inequality for kinetic {Fokker}-{Planck} equations}, J. Inst. Math. Jussieu
  \textbf{22} (2023), no.~6, 2749--2774 (English).

\bibitem[Gia94]{MR1288498}
Ugo Gianazza, \emph{Regularity for nonlinear equations involving square
  {H}\"{o}rmander operators}, Nonlinear Anal. \textbf{23} (1994), no.~1,
  49--73. \MR{1288498}

\bibitem[GIMV19]{gimv}
Fran\c{c}ois Golse, Cyril Imbert, Cl{\'e}ment Mouhot, and Alexis~F. Vasseur,
  \emph{{Harnack inequality for kinetic Fokker-Planck equations with rough
  coefficients and application to the Landau equation}}, Ann. Sc. Norm. Super.
  Pisa Cl. Sci. \textbf{19} (2019), no.~5, 253--295.

\bibitem[GM79]{GiaquintaModica}
M.~Giaquinta and G.~Modica, \emph{Regularity results for some classes of higher
  order non linear elliptic systems.}, Journal für die reine und angewandte
  Mathematik \textbf{0311\_0312} (1979), 145--169.

\bibitem[GM22]{zbMATH07559605}
Jessica Guerand and Cl{\'e}ment Mouhot, \emph{Quantitative {De} {Giorgi}
  methods in kinetic theory}, J. {\'E}c. Polytech., Math. \textbf{9} (2022),
  1159--1181 (English).

\bibitem[GS82]{GiaquintaStruwe}
Mariano Giaquinta and Michael Struwe, \emph{On the partial regularity of weak
  solutions of nonlinear parabolic systems}, Mathematische Zeitschrift
  \textbf{179} (1982), no.~4, 437--451.

\bibitem[GZ12]{MR2924890}
Luigi Greco and Gabriella Zecca, \emph{A version of {G}ehring lemma in {O}rlicz
  spaces}, Atti Accad. Naz. Lincei Rend. Lincei Mat. Appl. \textbf{23} (2012),
  no.~1, 29--50. \MR{2924890}

\bibitem[H\"67]{MR222474}
Lars H\"{o}rmander, \emph{Hypoelliptic second order differential equations},
  Acta Math. \textbf{119} (1967), 147--171. \MR{222474}

\bibitem[IS20]{ImbertSilvestre}
Cyril Imbert and Luis Silvestre, \emph{The weak {H}arnack inequality for the
  {B}oltzmann equation without cut-off}, J. Eur. Math. Soc. (JEMS) \textbf{22}
  (2020), no.~2, 507--592. \MR{4049224}

\bibitem[Iwa98]{Iwaniec98}
Tadeusz Iwaniec, \emph{The {G}ehring lemma}, Quasiconformal mappings and
  analysis ({A}nn {A}rbor, {MI}, 1995), Springer, New York, 1998, pp.~181--204.
  \MR{1488451}

\bibitem[Kin94]{zbMATH00691575}
Juha Kinnunen, \emph{Higher integrability with weights}, Ann. Acad. Sci. Fenn.,
  Ser. A I, Math. \textbf{19} (1994), no.~2, 355--366 (English).

\bibitem[KMS14]{MR3283394}
Tuomo Kuusi, Giuseppe Mingione, and Yannick Sire, \emph{A fractional {G}ehring
  lemma, with applications to nonlocal equations}, Atti Accad. Naz. Lincei
  Rend. Lincei Mat. Appl. \textbf{25} (2014), no.~4, 345--358. \MR{3283394}

\bibitem[KMS15]{MR3336922}
\bysame, \emph{Nonlocal self-improving properties}, Anal. PDE \textbf{8}
  (2015), no.~1, 57--114. \MR{3336922}

\bibitem[Kol34]{MR1503147}
A.~Kolmogoroff, \emph{Zuf\"{a}llige {B}ewegungen (zur {T}heorie der
  {B}rownschen {B}ewegung)}, Ann. of Math. (2) \textbf{35} (1934), no.~1,
  116--117. \MR{1503147}

\bibitem[ME75]{MR0417568}
Norman~G. Meyers and Alan Elcrat, \emph{Some results on regularity for
  solutions of non-linear elliptic systems and quasi-regular functions}, Duke
  Math. J. \textbf{42} (1975), 121--136. \MR{417568}

\bibitem[Mod85]{modica1985quasiminimi}
Giuseppe Modica, \emph{Quasiminimi di alcuni funzionali degeneri}, Annali di
  Matematica Pura ed Applicata \textbf{142} (1985), 121--143.

\bibitem[NW18]{MR3879985}
Pengcheng Niu and Huiju Wang, \emph{Gehring's lemma for {O}rlicz functions in
  metric measure spaces and higher integrability for convex integral
  funcationals}, Houston J. Math. \textbf{44} (2018), no.~3, 941--974.
  \MR{3879985}

\bibitem[PP04]{MR2068847}
Andrea Pascucci and Sergio Polidoro, \emph{The {M}oser's iterative method for a
  class of ultraparabolic equations}, Commun. Contemp. Math. \textbf{6} (2004),
  no.~3, 395--417. \MR{2068847}

\bibitem[SKP19]{MR4033780}
Samir~H. Saker, Mario Krni\'{c}, and Josip Pe\u{c}ari\'{c}, \emph{Higher
  summability theorems from the weighted reverse discrete inequalities}, Appl.
  Anal. Discrete Math. \textbf{13} (2019), no.~2, 423--439. \MR{4033780}

\bibitem[WZ09]{MR2530175}
WenDong Wang and LiQun Zhang, \emph{The {$C^\alpha$} regularity of a class of
  non-homogeneous ultraparabolic equations}, Sci. China Ser. A \textbf{52}
  (2009), no.~8, 1589--1606. \MR{2530175}

\bibitem[ZG05]{MR2173373}
Anna Zatorska-Goldstein, \emph{Very weak solutions of nonlinear subelliptic
  equations}, Ann. Acad. Sci. Fenn. Math. \textbf{30} (2005), no.~2, 407--436.
  \MR{2173373}

\end{thebibliography}

\end{document}